\documentclass[copyright,creativecommons]{eptcs}
\providecommand{\event}{ACT 2024}
\usepackage{graphicx} 
\usepackage{quiver}
\usepackage{verbatim}
\usepackage{todonotes}
\usepackage{xcolor}

\usepackage{amsmath, amsfonts, amssymb, amsthm, mathtools}
\usepackage{thmtools}
\usepackage{hyperref}
\DeclareMathAlphabet{\mathcal}{OMS}{cmsy}{m}{n}
\usepackage[ruled]{algorithm2e}

\newcommand{\CC}{\mathcal{C}}
\newcommand{\Li}{\mathcal{L}^D_i}
\newcommand{\Ld}{\mathcal{L}^D}
\newcommand{\Euc}{\mathrm{\textnormal{\textbf{Euc}}}}

\newcommand{\OptC}{\mathrm{\textnormal{\textbf{Opt}}}_\CC^R}
\newcommand{\OptEuc}{\mathrm{\textnormal{\textbf{Opt}}}_\Euc^\R}

\newcommand{\DynamC}{\mathrm{\textnormal{\textbf{Dynam}}}_\CC}

\newcommand{\Set}{\mathrm{\textnormal{\textbf{Set}}}}

\newcommand{\maps}{\colon}
\newcommand{\R}{\mathbb{R}}

\DeclareMathOperator{\Para}{Para}

\newcommand{\id}{\mathrm{id}}
\newcommand{\define}[1]{\textbf{#1}}

\newcommand{\gd}{\mathsf{gd}}
\newcommand{\GD}{\mathsf{GD}_\CC}
\newcommand{\GDEuc}{\mathsf{GD}_\Euc}
\newcommand{\rev}{\mathrm{\textnormal{\textbf{R}}}}
\newcommand{\fwd}{\mathrm{\textnormal{\textbf{D}}}}
\newcommand{\lin}{\mathsf{Lin}}
\newcommand{\op}{^\mathrm{op}}

\newcommand{\1}{\{*\}}

\newtheorem{theorem}{Theorem}[section]
\newtheorem{corollary}{Corollary}[theorem]
\newtheorem{lemma}[theorem]{Lemma}

\theoremstyle{definition}
\newtheorem{definition}[theorem]{Definition}
\newtheorem{rough}[theorem]{Rough Definition}

\title{Generalized Gradient Descent is a Hypergraph Functor}
\author{Tyler Hanks \qquad\qquad James Fairbanks
\institute{Department of Computer and Information Science and Engineering \\
University of Florida\\ Gainesville, Florida}
\email{t.hanks@ufl.edu \quad\qquad fairbanksj@ufl.edu}
\and
Matthew Klawonn 
\institute{Information Directorate \\ Air Force Research Lab \\ Rome, New York}
\email{matthew.klawonn.2@us.af.mil}
}
\def\titlerunning{Generalized Gradient Descent is a Hypergraph Functor}
\def\authorrunning{T. Hanks, M. Klawonn \& J. Fairbanks}

\begin{document}

\maketitle

\begin{abstract}



    Cartesian reverse derivative categories (CRDCs) provide an axiomatic generalization of the reverse derivative, which allows generalized analogues of classic optimization algorithms such as gradient descent to be applied to a broad class of problems.
    In this paper, we show that generalized gradient descent with respect to a given CRDC induces a hypergraph functor from a hypergraph category of optimization problems to a hypergraph category of dynamical systems. The domain of this functor consists of objective functions that are 1) general in the sense that they are defined with respect to an arbitrary CRDC, and 2) open in that they are decorated spans that can be composed with other such objective functions via variable sharing. The codomain is specified analogously as a category of general and open dynamical systems for the underlying CRDC. We describe how the hypergraph functor induces a distributed optimization algorithm for arbitrary composite problems specified in the domain. To illustrate the kinds of problems our framework can model, we show that parameter sharing models in multitask learning, a prevalent machine learning paradigm, yield a composite optimization problem for a given choice of CRDC. We then apply the gradient descent functor to this composite problem and describe the resulting distributed gradient descent algorithm for training parameter sharing models.
    
    
    
\end{abstract}
\section{Introduction}\label{sec:intro}

\begin{figure}\label{fig:mtl}
    \centering
    \includegraphics[width=\textwidth]{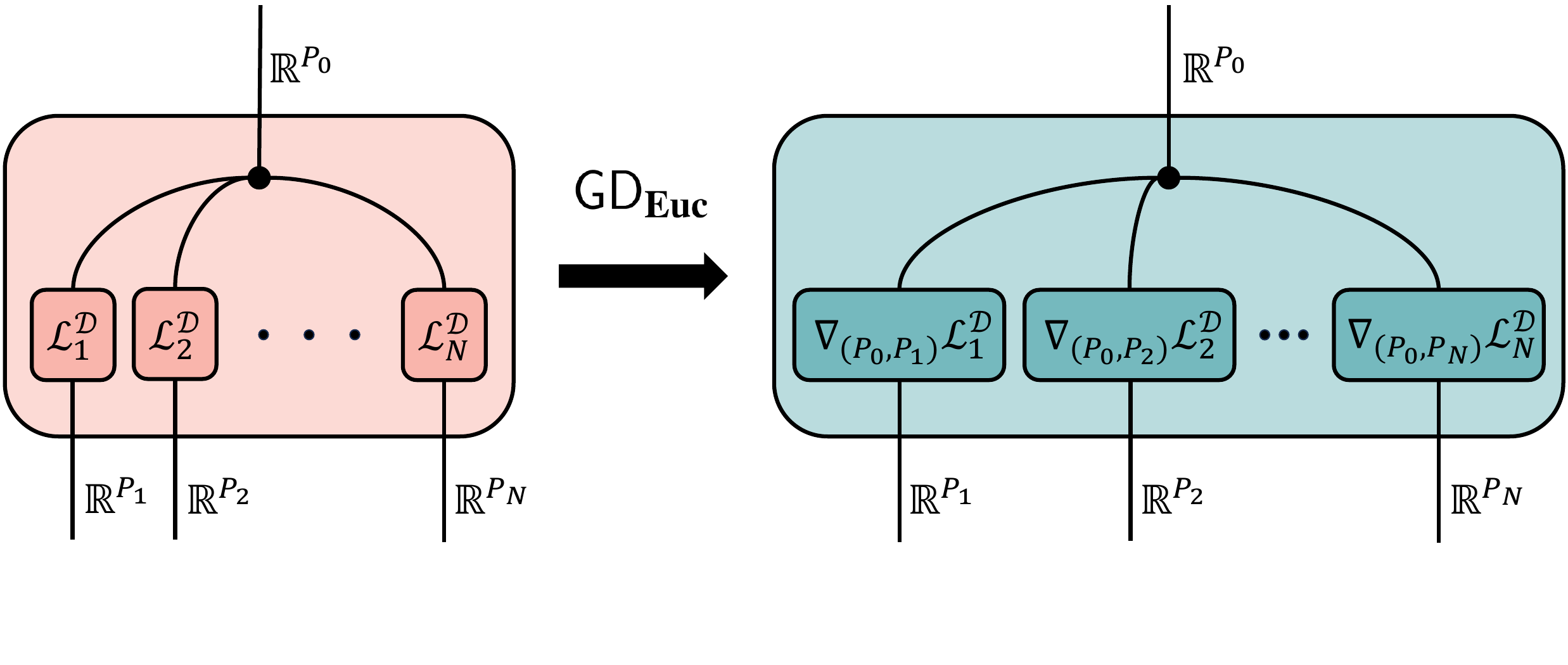}
    \caption{An illustration of multitask learning in our framework. The left string diagram shows the composite objective function representing the multitask learning objective. The right string diagram is the result of applying the gradient descent functor to the morphism on the left, resulting in a distributed optimization algorithm.
    }
    \label{fig:string_diagram}
\end{figure}

Cartesian reverse derivative categories (CRDCs) provide an axiomatic generalization of the reverse derivative and have been used to define generalized analogues of classic optimization algorithms, such as gradient descent, for minimizing generalized objective functions \cite{cruttwell2022categorical, shiebler_generalized_2022}. These developments allow techniques from machine learning such as backpropagation to be applied not only to artificial neural networks, but to a broad class of models including boolean circuits. 
The generality of the CRDC framework makes future expansion to other problem types likely. 
Central to the definition of a CRDC is the Cartesian reverse derivative combinator $\rev$, which must obey axioms mirroring the behavior of the standard directional derivative in the Euclidean domain. This combinator is \emph{compositional} in that it satisfies a chain rule, allowing one to define a general version of the backpropagation algorithm \cite{rumelhart1986learning}.

In prior work, CRDCs have been used primarily to compose parameterized morphisms, i.e to build a learning model, all while leveraging $\rev$ to ensure the parameters can be updated with respect to the generalized gradient of an objective. In this way, the induced optimization problem of maximizing/minimizing the given objective with respect to the parameters has a compositional structure by virtue of the \emph{model} being compositional. We propose that in addition to compositional structure in the \emph{model}, compositional structure in the \emph{objective} is also of interest. In particular, it is common in machine learning for parameters to be optimized for more than one objective. Simple examples of compositional objectives are those arising from various regularization methods, for example $\ell_1$ or $\ell_2$ penalties \cite{bishop2006pattern}, wherein one term in the objective incentivizes ``good performance" on the given task while the other enforces some generically desirable property like sparseness. Other examples with more sophisticated composition patterns include example- and feature-parallel distributed learning setups \cite{boyd_distributed_2010}, and parameter sharing methods wherein subsets of parameters may be optimized for one or more objective functions. Parameter sharing arises within the setting of multitask learning, which we will discuss in this paper. Figure \ref{fig:mtl} illustrates this example. Our goal in this paper is to formalize compositional objective structure and its exploitation via gradient descent in the general setting of CRDCs.

To capture composition of objective functions that may have shared parameters, we turn to decorated spans and hypergraph categories, which are useful for formalizing mathematical objects that compose in an undirected, networked fashion \cite{fong2016algebra}. Our approach closely follows that of \cite{hanks2024compositional}, where algebras of undirected wiring diagrams are used to compose objective functions with shared decision variables, and a variety of first-order methods including (sub)gradient descent and primal-dual methods are shown to be functorial in the standard Euclidean domain. This paper extends the results on gradient descent to the general setting of Cartesian reverse derivative categories and generalized optimization introduced in \cite{shiebler_generalized_2022}. Specifically, our contributions are as follows.

\begin{itemize}
    \item We show how to produce hypergraph categories $\OptC$ of generalized open objectives and $\DynamC$ of optimizers over a given CRDC $\CC$ and optimization domain $(\CC,R)$.
    \item We prove that generalized gradient descent yields a hypergraph functor
    \[
    \GD\maps \OptC\to \DynamC
    \]
    from open objectives to open optimizers.
    \item We highlight how functoriality allows the gradient descent solution algorithms for composite problems to be implemented in a distributed fashion.
    \item We show that multitask learning \cite{caruana1997multitask} with hard parameter sharing is an example of objective composition occurring in the machine learning literature, and we leverage our functorial formulation of gradient descent to recover a distributed training scheme.
\end{itemize}

After providing the necessary background on CRDCs, we will begin by setting up the categories of open generalized objectives and open dynamical systems. We then specify the generalized gradient descent functor between them before turning to the multitask learning example.

\section{Preliminaries}

In this section, we recall the necessary definitions and concepts from Cartesian reverse derivative categories \cite{cockett2019reverse} and generalized optimization \cite{shiebler_generalized_2022}. 

\begin{definition}[Definition 3.1 in \cite{shiebler_generalized_2022}]
    A \define{Cartesian left-additive category} is a Cartesian category $\CC$ with terminal object $*$ in which for any pair of objects $X,Y$, the hom-set $\CC(X,Y)$ is a commutative monoid with addition operation $+$ and additive identities $0_{X,Y}\maps X\to Y$. In addition, the following axioms must be satisfied.
    \begin{itemize}
        \item \textbf{LA.1} For any morphisms $h\maps X\to Y$ and $f,g\maps Y\to Z$, we have
        \begin{equation}
            (f+g)\circ h = (f\circ h) + (g\circ h)\maps X\to Z\text{  and  } 0_{Y,Z}\circ h = 0_{X,Z}\maps X\to Z.
        \end{equation}
        \item \textbf{LA.2} For any projection map $\pi_i\maps Y\to Z$ and morphisms $f,g\maps X\to Y$ we have
        \begin{equation}
            \pi_i\circ(f+g) = (\pi_i\circ f) + (\pi_i\circ g)\maps X\to Z\text{ and } \pi_i\circ 0_{X,Y} = 0_{X,Z}\maps X\to Z. 
        \end{equation}
    \end{itemize}
    The additive identity of $\CC(*, X)$ is denoted $0_X$.
\end{definition}

\begin{rough}[Definition 3.3 in \cite{shiebler_generalized_2022}]
    A \define{Cartesian differential category} is a Cartesian left-additive category $\CC$ equipped with a \define{Cartesian derivative combinator} $\fwd$ taking morphisms $f\maps X\to Y$ to morphisms $\fwd[f]\maps X\times X\to Y$. This combinator must satisfy several axioms requiring that it behave like a standard forward derivative.
\end{rough}

\begin{rough}[Definition 3.2 in \cite{shiebler_generalized_2022}]
    A \define{Cartesian reverse derivative category} is a Cartesian left-additive category $\CC$ equipped with a \define{Cartesian reverse derivative combinator} $\rev$ taking morphisms $f\maps X\to Y$ to morphisms $\rev[f]\maps X\times Y\to X$. This combinator must satisfy several axioms requiring that it behave like a standard reverse derivative. In particular, we will make use of the following axioms in this paper.
    \begin{itemize}
        \item \textbf{RD.1 } $\rev[f+g] = \rev[f] + \rev[g]$ and $\rev[0] = 0$.
        \item \textbf{RD.5 } $\rev[g\circ f]=\rev[f]\circ(\id\times\rev[g])\circ\langle\pi_0,\langle f\circ\pi_0,\pi_1\rangle\rangle$.
    \end{itemize}
\end{rough}
Intuitively, \textbf{RD.1} says that the reverse derivative combinator must be additive while \textbf{RD.5} is the chain rule cast in the abstract setting of CRDCs. Every Cartesian reverse derivative category (CRDC) is also a Cartesian differential category (Theorem 16 in \cite{cockett2019reverse}) by defining the derivative combinator of a morphism $f\maps X\to Y$ as
\begin{equation}
    \fwd[f] = \pi_1\circ \rev[\rev[f]]\circ(\langle\id_X,0_{X,Y}\rangle\times\id_X)\maps X\times X\to Y.
\end{equation}
This allows us to refer to the Cartesian derivative combinator of a CRDC. Of crucial importance in the subsequent sections is the notion of a linear map in a CRDC.

\begin{definition}
    A \define{linear map} in a CRDC $\CC$ is a morphism $f$ such that $\fwd[f]=f\circ\pi_1$. The linear maps in $\CC$ form a full subcategory of $\CC$ which we denote $\lin(\CC)$. By Proposition 24 in \cite{cockett2019reverse}, $\lin(\CC)$ is a $\dag$-category with finite $\dag$-biproducts. Moreover, $\rev[f]=f^\dag\circ\pi_1$ and $(f+g)^\dag = f^\dag + g^\dag$.
\end{definition}

There are two canonical examples of CRDCs in the literature.
\begin{enumerate}
    \item The CRDC $\Euc$ has Euclidean spaces as objects and smooth functions as morphisms. The terminal object is $\R^0$ and the reverse derivative of a smooth function $f\maps \R^m\to\R^n$ is 
    \[
        \rev[f](x,x')\coloneqq J_f(x)^Tx',
    \]
    where $J_f(x)$ is the Jacobian of $f$ evaluated at $x$. The linear maps in $\Euc$ are the usual linear maps between vector spaces, with the dagger structure corresponding to dual maps.
    \item Given a commutative ring $r$, the CRDC $\textbf{Poly}_r$ has natural numbers as objects. A morphism from $n$ to $m$ is an $m$-tuple of polynomials over $r$, each in $n$ variables. The terminal object is $0$ and the reverse derivative of a polynomial $P(x)=(p_1(x),\dots,p_m(x))$ is
    \[
        \rev[P](x,x') \coloneqq \Bigl(\sum_{i=1}^m\frac{\partial p_i}{\partial x_1}(x)x_i',\dots,\sum_{i=1}^m\frac{\partial p_i}{\partial x_n}(x)x_i'\Bigr),
    \]
    where $\frac{\partial p_i}{\partial x_j}(x)$ denotes the formal derivative of $p_i$ in $x_j$ at $x$. The linear maps are polynomials $P(x)=(p_1(x),\dots,p_m(x))$ where each $p_i(x)$ is of the form $p_i(x) = \sum_{i=1}^n r_ix_i$ for $r_i\in r$ \cite{cockett2019reverse}.
\end{enumerate}
The fundamental component for generalized optimization is the notion of an \emph{optimization domain}, which generalizes the role of the real numbers in defining the ``cost" or ``value" of a given configuration of decision variables.
\begin{definition}[Definition 3.4 in \cite{shiebler_generalized_2022}]\label{def:opt-dom}
    An \define{optimization domain} is a pair $(\CC,R)$ where $\CC$ is a CRDC and $R$ is an object of $\CC$ equipped with the following additional structures.
    \begin{itemize}
        \item Each morphism $f\maps X\to Y$ in $\CC$ has an additive inverse $-f$.
        \item Each hom-set $\CC(*,X)$ out of the terminal object is equipped with a multiplication operation $fg$ and a multiplicative identity $1_X\maps *\to X$ to form a commutative ring with the left additive structure.
        \item The hom-set $\CC(*,R)$ is totally ordered to form an ordered commutative ring.
    \end{itemize}
    Given a unique map $!_X\maps X\to *$ into the terminal object, the map $1_Y\circ !_X\maps X\to Y$ is denoted $1_{X,Y}$.
    An \define{objective} in $(\CC,R)$ is a morphism $\ell\maps X\to R$ in $\CC$.
\end{definition}

\begin{definition}[Definition 3.6 in \cite{shiebler_generalized_2022}]
    Given an optimization domain $(\CC,R)$, the \define{generalized gradient} of an objective $\ell\maps X\to R$ is $\rev[\ell]_1\maps X\to X$ defined by
    \begin{equation}
        \rev[\ell]_1\coloneqq \rev[\ell]\circ\langle\id_X,1_{X,R}\rangle.
    \end{equation}
\end{definition}

The following are examples given in \cite{shiebler_generalized_2022} of optimization domains and generalized gradients.
\begin{enumerate}
    \item The \define{standard domain} is $(\Euc,\R)$. Objectives are smooth functions $\ell\maps\R^n\to\R$ and the gradient of $\ell\maps\R^n\to\R$ is $\nabla\ell\maps\R^n\to\R^n$.
    \item The \define{$r$-polynomial domain} is $(\textbf{Poly}_r, 1)$. Objectives are polynomials $\ell\maps n\to 1$ and the gradient of $\ell\maps n\to 1$ is $\Bigl(\frac{\partial\ell}{\partial x_1}(x),\dots,\frac{\partial\ell}{\partial x_n}(x)\Bigr)$, where these are again formal derivatives.
\end{enumerate}

\section{Generalized Gradient Descent as a Hypergraph Functor}\label{sec:GGD}
To construct hypergraph categories of objectives and optimizers, we will require that the subcategory of linear maps in $\CC$, denoted $\lin(\CC)$, has finite limits. Going forward, we use $\times$ to denote the product and $*$ to denote the terminal object in an arbitrary CRDC. Proofs of theorems are given in the appendix.

\subsection{The Hypergraph Category of Open Objectives}

The goal of this section is to define a hypergraph category of generalized optimization problems. Recall that a hypergraph category is a symmetric monoidal category where every object is equipped with the structure of a \emph{special commutative Frobenius monoid} \cite{fong_hypergraph_2019}. In particular, this means that every object has the structure of both a commutative monoid and a cocommutative comonoid and satisfies two additional axioms to ensure that these structures cohere properly. It can be further shown from these axioms that every hypergraph category is self-dual compact closed. Taken together, these properties yield a category where morphisms can compose in an undirected fashion according to the interconnection pattern of arbitrary hypergraphs (hence the name hypergraph categories). Another way of saying this is that hypergraph categories are symmetric monoidal categories whose string diagrams (introduced in \cite{JOYAL199320}) are hypergraphs. We now turn to how this structure is realized in optimization problems.

In standard optimization, a generic unconstrained minimization problem has the form
\[
    \text{minimize } f(x),
\]
where $f\maps\R^n\to\R$ is known as the objective function, and $x\in\R^n$ is referred to as the decision or optimization variable. When $f(x)$ has the form $\sum_{i=1}^N f_i(x_i)$, where the $x_i$ are disjoint subvectors of $x$, $f$ is referred to as \emph{separable}. Separable objectives can be solved entirely in parallel as each optimal $x_i$ has no effect on the rest of the objective. However, in many important problems, the objective function is \emph{partially separable}. For example, the problem
\[
    \text{minimize } f(w,x) + g(u,w,y) + h(u,w,z)
\]
consists of a sum of objectives with \emph{complicating} variables $u$ and $w$. Such a sum of objectives with shared variables is the type of compositional structure that we wish to capture and generalize in our hypergraph category of optimization problems.

A standard method for defining such a hypergraph category is to leverage the theory of decorated spans\footnote{The literature typically uses \emph{covariant} decorating functors to define decorated cospan categories. However, all the same results hold for the dual construction of \emph{contravariant} decorating functors and decorated spans, which we leverage in this paper.} \cite{fong2015decorated, fong2016algebra}. Decorated spans turn a closed system into an open system by specifying a boundary for the system in addition to morphisms that relate the domain of the system to its boundary. We now leverage this general pattern to define a notion of \emph{open objectives}.

\begin{definition}
    Given an optimization domain $(\CC,R)$, an \define{open objective} with domain boundary $X$ and codomain boundary $Y$ consists of a span $X\leftarrow N\rightarrow Y$ in $\lin(\CC)$ together with a choice of objective on $N$, i.e., a map $f\maps N\to R$ in $\CC$. This is summarized by the diagram
\[\begin{tikzcd}
	& R \\
	X & N & Y.
	\arrow["f", from=2-2, to=1-2]
	\arrow["l"', from=2-2, to=2-1]
	\arrow["r", from=2-2, to=2-3]
\end{tikzcd}\]
\end{definition}

There is a straightforward interpretation of open objectives when the legs are both projections: the boundary objects are subobjects of the overall decision space which can be shared with other objectives. This intuition of sharing parts of the decision space is still helpful even when the legs are general linear maps. The theory of decorated spans provides a recipe for constructing hypergraph categories of such open systems by specifying a lax symmetric monoidal functor from a finitely complete category into $\Set$, known as a \emph{decorating functor}. We thus arrange open objectives into a hypergraph category by specifying the following decorating functor.

\begin{restatable}{theorem}{obj}\label{thm:obj}
    Given an optimization domain $(\CC,R)$, there is a contravariant lax symmetric monoidal functor $O_\CC^R\maps (\lin(\CC)\op,\times)\to (\Set,\times)$ defined by the following maps (where we let $O\coloneqq O_\CC^R$ for the remainder of the theorem and proof).
    \begin{itemize}
        \item Given an object $X$, $O(X)$ is the set of generalized objectives with domain $X$, i.e., the hom-set $\CC(X,R)$.
        \item Given a linear map $\phi\maps X\to Y$ and objective $f\maps Y\to R$, $O(\phi)(f)\maps X\to R$ is the objective $f\circ \phi$ obtained by precomposition.
        \item Given objects $X,Y$, the product comparison $\varphi_{X,Y}\maps O(X)\times O(Y)\to O(X\times Y)$ is defined by
        \begin{equation}
            \varphi_{X,Y}(f,g)\coloneqq f\circ \pi_0 + g\circ \pi_1,
        \end{equation}
        where $\pi_0\maps X\times Y\to X$ and $\pi_1\maps X\times Y\to Y$ are the natural projections.
        \item The unit comparison $\varphi_0\maps \1\to O(*)$ selects the additive identity $0_R$ for the hom-set $\CC(*,R)$ supplied by the left additive structure.
    \end{itemize}
\end{restatable}

We can now use our decorated spans of open objectives to construct a hypergraph category.

\begin{corollary}[Open Generalized Optimization Problems] \label{cor:Obj} 
    Given an optimization domain $(\CC,R)$ such that $\lin(\CC)$ has finite limits, there is a hypergraph category $\OptC$ defined by the following data.
    \begin{itemize}
        \item Objects are the same as those of $\CC$.
        \item Morphisms are (isomorphism classes of) open generalized objectives.
        \item Given two open objectives $(X\xleftarrow{l_1} N\xrightarrow{r_1}Y, f\in O_\CC^R(N))$ and $(Y\xleftarrow{l_2}M\xrightarrow{r_2}Z, g\in O_\CC^R(M))$, their composite consists of the following span computed by pullback
        \begin{equation}\label{eq:pullback}
\begin{tikzcd}
	&& {N\times_YM} \\
	\\
	& N & {N\times M} & M \\
	X && Y && Z
	\arrow["{l_1}", from=3-2, to=4-1]
	\arrow["{r_1}"', from=3-2, to=4-3]
	\arrow["{l_2}", from=3-4, to=4-3]
	\arrow["{r_2}"', from=3-4, to=4-5]
	\arrow["{\pi_1}", from=3-3, to=3-4]
	\arrow["{\pi_0}"', from=3-3, to=3-2]
	\arrow["{b_1}", curve={height=-12pt}, from=1-3, to=3-4]
	\arrow["{b_0}"', curve={height=12pt}, from=1-3, to=3-2]
	\arrow["{\langle b_0,b_1\rangle}"{description}, dashed, from=1-3, to=3-3]
\end{tikzcd}
    \end{equation}
    together with the objective obtained from the composite morphism
    \begin{equation}
        \1\cong \1\times \1\xrightarrow{f\times g}O_\CC^R(N)\times O_\CC^R(M)\xrightarrow{\varphi_{N,M}}O_\CC^R(N\times M)\xrightarrow{O_\CC^R\langle b_0,b_1\rangle }O_\CC^R(N\times_Y M).
    \end{equation}
    \item The identity on $X$ is the pair $(X=X=X, \1\xrightarrow{\varphi_0}O_\CC^R(*)\xrightarrow{O_\CC^R(!)}O_\CC^R(X))$.
    \item The monoidal product on objects is their product in $\CC$ while the monoidal product of morphisms $(X\leftarrow N\to Y, f)$ and $(W\leftarrow M\to Z, g)$ is
    \begin{equation}\label{eq:obj_mon}
        (X\times W\leftarrow N\times M\to Y\times Z, \varphi_{N,M}(f,g)).
    \end{equation}
    \item The hypergraph maps are inherited from those of $\mathrm{Span}(\lin(\CC))$ together with the identity decoration.
    \end{itemize}
\end{corollary}
\begin{proof}
    This follows by applying Proposition 3.2 and Theorem 3.4 in \cite{fong2015decorated} to Theorem \ref{thm:obj}.
\end{proof}

\subsection{The Hypergraph Category of Open Optimizers}
Having defined a hypergraph category of open objective functions, we will now use the same machinery of decorated spans to define a hypergraph category of open dynamical systems over a given CRDC. 

\begin{definition}
    Given a CRDC $\CC$, an \define{open dynamical system} with domain boundary $X$ and codomain boundary $Y$ consists of a span $X\leftarrow N\rightarrow Y$ in $\lin(\CC)$ together with a choice of system on $N$, i.e., a map $v\maps N\to N$ in $\CC$.
\end{definition}

Endomaps in a CRDC are also referred to as optimizers in \cite{shiebler_generalized_2022}, thus we use the terms optimizer and dynamical system interchangeably. Similar to open objectives, the boundaries of open dynamical systems specify which parts of a system's state space can be influenced by other systems. The decorating functor is given as follows.

\begin{restatable}{theorem}{opt}\label{thm:opt}
    Given a CRDC $\CC$, there is a contravariant lax symmetric monoidal functor \\$D_\CC\maps (\lin(\CC)\op,\times)\to (\Set,\times)$ defined by the following maps (where we let $D\coloneqq D_\CC$ for the remainder of the theorem and proof).
    \begin{itemize}
        \item Given an object $X$, $D(X)$ is the set of generalized dynamical systems with state space $X$, i.e., the hom-set $\CC(X,X)$.
        \item Given a linear map $\phi\maps X\to Y$ and a system $v\maps Y\to Y$, $D(\phi)(v)$ is the system $\phi^\dag\circ v\circ \phi\maps X\to X$.
        \item Given objects $X,Y$, the product comparison $\varphi_{X,Y}\maps D(X)\times D(Y)\to D(X\times Y)$ is defined by
        \begin{equation}
            \varphi_{X,Y}(v,w)\coloneqq \pi_0^\dag\circ v\circ \pi_0 + \pi_1^\dag\circ w\circ \pi_1.
        \end{equation}
        \item The unit comparison $\varphi_0\maps \1\to D(*)$ is uniquely determined as $\CC(*,*)$ is singleton.
    \end{itemize}
\end{restatable}

This is a generalization of the dynamics decorating functor defined in \cite{baez2017compositional}. We can now define an analogous hypergraph category of open generalized dynamical systems.

\begin{corollary}[Open Generalized Dynamics]
    Given a CRDC $\CC$ such that $\lin(\CC)$ has finite limits, there is a hypergraph category $\DynamC$ defined by the following data.
    \begin{itemize}
        \item Objects are the same as those of $\CC$.
        \item Morphisms are (isomorphism classes of) open dynamical systems.
        \item Given two open systems $(X\xleftarrow{l_1} N\xrightarrow{r_1}Y, v\in D_\CC(N))$ and $(Y\xleftarrow{l_2}M\xrightarrow{r_2}Z, w\in D_\CC(M))$, their composite consists of the pullback span as in \eqref{eq:pullback} together with the system obtained from the composite morphism
    \begin{equation}
        \1\cong \1\times \1\xrightarrow{v\times w}D_\CC(N)\times D_\CC(M)\xrightarrow{\varphi_{N,M}}D_\CC(N\times M)\xrightarrow{D_\CC\langle b_0,b_1\rangle }D_\CC(N\times_Y M).
    \end{equation}
    \item The identity on $X$ is the pair $(X=X=X, \1\xrightarrow{\varphi_0}D_\CC(*)\xrightarrow{D_\CC(!)}D_\CC(X))$.
    \item The monoidal product on objects is their product in $\CC$ while the monoidal product of morphisms $(X\leftarrow N\to Y, v)$ and $(W\leftarrow M\to Z, w)$ is
    \begin{equation}
        (X\times W\leftarrow N\times M\to Y\times Z, \varphi_{N,M}(v,w)).
    \end{equation}
    \item The hypergraph maps are inherited from those of $\mathrm{Span}(\lin(\CC))$ together with the identity decoration.
    \end{itemize}
\end{corollary}
\begin{proof}
    This follows by applying Proposition 3.2 and Theorem 3.4 in \cite{fong2015decorated} to Theorem \ref{thm:opt}.
\end{proof}

\subsection{Functoriality of Generalized Gradient Descent}
The pieces are now in place for our main result. We would like to show that one can functorially relate a composite objective function defined in $\OptC$ to a composite optimizer in $\DynamC$ that performs gradient descent on the given objective function. The framework of decorated spans allows the construction of such a hypergraph functor by specifying a monoidal natural transformation between the underlying decorating functors defining the domain and codomain hypergraph categories. Note that below we work with continuous gradient flow systems to simplify proofs. This is acceptable as gradient descent is equivalent to Euler's discretization method applied to a gradient flow system. Therefore, these results extend to discrete gradient descent systems because Euler's method is provably functorial in the Euclidean domain \cite{libkind_operadic_2022}, and this proof can be easily extended to the general CRDC domain.


\begin{restatable}{theorem}{ggd}\label{thm:ggd}
    Given an optimization domain $(\CC, R)$, there is a monoidal natural transformation \\$\gd(\CC)\maps O^R_\CC\Rightarrow D_\CC$ with components $\gd(\CC)_X\maps O_\CC^R(X)\to D_\CC(X)$ defined by the generalized gradient descent optimization scheme, i.e.,
    \begin{equation}
        f\mapsto -\rev[f]_1.
    \end{equation}
\end{restatable}

\begin{corollary}[Functoriality of Generalized Gradient Flow]\label{thm:GD}
    Given an optimization domain $(\CC,R)$ such that $\lin(\CC)$ has finite limits, there is an identity on objects hypergraph functor $\GD\maps\OptC\to \DynamC$ which takes an open objective $(X\xleftarrow{l}N\xrightarrow{r}Y, f\maps N\to R)$ to the open optimizer $(X\xleftarrow{l}N\xrightarrow{r}Y, -\rev[f]_1\maps N\to N)$.
\end{corollary}
\begin{proof}
    This follows by applying Theorem 4.1 of \cite{fong2015decorated} to the monoidal natural transformation $\gd$.
\end{proof}

Taking stock, there are a many implications of the categories and functor just defined. For one, corollary \ref{thm:GD} says we have the following equality given any composable pair $F\coloneqq(X\leftarrow N\to Y, f\maps N\to R),G\coloneqq(Y\leftarrow M\to Z, g\maps M\to R)$ of open objectives:
\begin{equation}
    \GD(G\circ F) = \GD(G)\circ\GD(F).
\end{equation}
In particular, this yields the following equivalence of open optimizers:
\begin{equation}\label{eq:optimizer_comp}
    -\rev[f\circ\pi_0\circ\phi + g\circ\pi_1\circ\phi]_1=-\phi^\dag\circ(\iota_0\circ\rev[f]_1\circ\pi_0 + \iota_1\circ\rev[g]_1\circ\pi_1)\circ\phi,
\end{equation}
where we are using $\phi$ to denote the pairing $\langle b_0,b_1\rangle$ from \eqref{eq:pullback}. Although these two optimizers are \emph{extensionally} equivalent (i.e., they compute the same output for a given input), they are not \emph{computationally} equivalent. In particular, the optimizer on the RHS of \eqref{eq:optimizer_comp} is far better suited to implementation in a distributed computing environment as it can be interpreted with message passing semantics. Specifically, the distribute step is given by applying the linear map $\phi$ to the input followed by the application of the projections to send the desired components to the subsystems. The parallel computation step is given by computing the gradients of each objective with respect to their current local values, and the collect step is given by injecting the results into a single vector and applying the linear map $\phi^\dag$.

Another benefit of open objectives and optimizers is the ability to easily specify composite objective functions using the graphical syntax of string diagrams. In the following section we will model hard parameter sharing for mulit-task learning in our category $\OptC$. We do so to show that, rather than existing as an exercise in abstraction, our framework is applicable to a prevalent machine learning paradigm. 

\section{Multitask Learning}
Multitask learning (MTL) \cite{caruana1997multitask} is a machine learning paradigm in which learning signals from different tasks are aggregated in order to train a single model. Precise definitions of what a task is vary, but for our purposes we will consider a task $\mathcal{T}_i \coloneqq (\mathcal{D}_i, \mathcal{L}_i)$ to be a pair of data set $\mathcal{D}_i$ and loss function $\mathcal{L}_i$. Usually (see \cite{zhang2021survey}) the tasks $\mathcal{T}_i$ are supervised, so $\mathcal{D}_i \coloneqq \mathcal{X}_i \times \mathcal{Y}_i$ for input data $\mathcal{X}_i$ and labels $\mathcal{Y}_i$. The primary (or at least original) motivation of MTL is to exploit additional learning signals for improved model performance on a given ``main" task, beyond what might be possible if the model were trained on any single task \cite{caruana1997multitask}.

A common approach (see \cite{zhang2021survey}) to multitask learning is to leverage parameterised models where at least some parameters are trained on all tasks. Such an approach is called \emph{hard parameter sharing} because the shared parameters are updated \emph{directly} using gradient information from multiple tasks. This is in contrast to \emph{soft parameter sharing} where the parameters are only \emph{penalized} for differing from counterparts trained on other tasks. Theoretically, hard parameter sharing helps to avoid overfitting of shared parameters \cite{baxter1997bayesian}. When the underlying model is a neural network, hard parameter sharing is often achieved by sharing the parameters of the first several layers while training a unique set of predictive final layers per task. However, recent work suggests that more flexible schemes may lead to improved performance \cite{zhang2022rethinking}. With this motivation, we will now frame hard parameter sharing as a composite optimization problem that allows arbitrary parameters to be shared amongst arbitrary tasks.

For each task $\mathcal{T}_i$ in a collection of $N$ tasks, consider a neural network $f_i$ that has task-specific parameters $P_i$ and shared parameters $P_0$. We will call an assignment of parameters to real numbers $W_i \in \R^{P_i}$ the \emph{weights} of a neural network, though note that the machine learning literature often uses the terms ``weights" and ``parameters" interchangeably. The composite optimization problem given by such a multitask learning setup (cf. \cite{sener2018multi}, (1)) is


\begin{equation}\label{eq:mtl_gd}
    \begin{array}{lll}
    \min\limits_{W_0, \ldots, W_N} & \sum_{i = 1}^{N} & \Li(f_i; W_0, W_i) \\
    \end{array}
\end{equation}

where
\begin{equation}
    \begin{array}{ll}
         \Li(f_i; W_0, W_i) & \coloneqq \frac{1}{|\mathcal{D}_i|} \sum_{d = 1}^{|\mathcal{D}^i|} \mathcal{L}_i(f_i(\mathcal{X}^d_i; W_0, W_i), \mathcal{Y}_i^d)  \\
    \end{array}
\end{equation}

with $f_i(\mathcal{X}_i^d; W_0, W_i)$ denoting the prediction of the network $f_i$ weighted by $(W_0, W_i)$ on input datum $\mathcal{X}_i^d$. We will now describe how the problem in \eqref{eq:mtl_gd} is a composite of open objectives, and therefore how to mechanically recover a distributed gradient descent algorithm for it.

\subsection{Compositional Formulation}
We will take our CRDC $\CC$ to be $\Euc$ with objects euclidean spaces $\R^N$ and morphisms $f: \R^N \rightarrow \R^M$ the smooth maps between spaces. The monoidal product is given by $\times$, the standard product of Euclidean spaces. We make this choice for consistency with the standard formulation above; however, we emphasize that our machinery is generic to the choice of CRDC. This means we could just as easily derive a distributed MTL training algorithm for some other non-standard CRDC. Now, each task's loss is a smooth map $\Li \maps \R^{P_i} \times \R^{P_0} \rightarrow \R$, meaning our optimization domain is $(\Euc, \R)$. Note that we have avoided some difficulty by treating each loss as an atomic map, as opposed to the composition of a neural network applied to data followed by a loss function. In the latter case, we would need to consider the set of parameters for the network in addition to the set of data points, and handle the subtleties of updating parameters while ignoring gradients w.r.t the input data. See \cite{cruttwell2022categorical} for a thorough discussion of such concerns. Because the loss $\Li$ is computed with respect to the entire data set $\mathcal{D}_i$ for a task, we have restricted ourselves to optimization methods that perform true gradient descent as opposed to the stochastic or mini-batch variants which are more common in the machine learning literature. We leave an extension of our work to the SGD setting for future efforts.

With the goal of specifying problem \eqref{eq:mtl_gd} as a morphism of $\OptEuc$, we first consider the following span that relates shared and task specific parameters for some $i$.

\begin{equation}\label{eq:li_span}
    \R^{P_0} \xleftarrow{\pi_0} \R^{P_0} \times \R^{P_i} \xrightarrow{\pi_1} \R^{P_i} 
\end{equation}

To promote \eqref{eq:li_span} to a decorated span, we need only to specify an objective function $\R^{P_0} \times \R^{P_i} \rightarrow \R$. With $\Li$ being the obvious choice, we construct a decorated span $F^i \coloneqq \left( \R^{P_0} \xleftarrow{\pi_0} \R^{P_0} \times \R^{P_i} \xrightarrow{\pi_1} \R^{P_i}, \Li \right)$. Each $F^i$ is a morphism in the hypergraph category $\OptEuc$, meaning that we may take the monoidal product $\bigotimes_{i = 1}^N F^i$. By Cor. \ref{cor:Obj}, specifically \eqref{eq:obj_mon}, we know that the monoidal product of our task specific decorated spans is as follows.

\begin{equation}
    \bigotimes_{i = 1}^N F^i = \left(\prod_{i = 1}^N \R^{P_0} \xleftarrow{\prod_{i=1}^N \pi^i_0} \prod_{i=1}^N (\R^{P_0} \times \R^{P_i}) \xrightarrow{\prod_{i=1}^N \pi_1^i} \prod_{i=1}^N \R^{P_i}, \varphi(\Ld_1, \ldots, \Ld_N) \right)
\end{equation}

Note that by Thm. \ref{thm:obj} the decorating objective function is given by the comparison map:
\begin{equation}
    \varphi(\Ld_1, \ldots, \Ld_N) \maps \prod_{i=1}^N O(\R^{P_0} \times \R^{P_i}) \rightarrow O(\prod_{i=1}^N \R^{P_0} \times \R^{P_i}),    
\end{equation}
where $\varphi(\Ld_1, \ldots, \Ld_N)$ maps the tuple of task specific losses to their sum:
\begin{equation}
    (\Ld_1, \ldots, \Ld_N) \mapsto \sum_{i=1}^N  \Li \circ \pi^i,
\end{equation}
where $\pi^i$ denotes the $i$'th projection of $\prod_{i=1}^N(\R^{P_0}\times\R^{P_i})$. We now wish to specify that the parameter space $\R^{P_0}$ is shared among all $N$ loss functions. This is accomplished as follows.

Recall that in a hypergraph category like $\OptEuc$, we have comultiplication maps 
\begin{equation}
    \delta_p \coloneqq \left( \R^p \xleftarrow{\id_p} \R^p \xrightarrow{\Delta} \R^p \times \R^p, 0 \maps \R^p \rightarrow \R \right),
\end{equation}
where $0\maps\R^p\to \R$ is the ``empty" decoration, i.e., the constant 0 objective function $0(v) = 0$. The map $\Delta\maps\R^p\to\R^p\times\R^p$ is the duplication map $v\mapsto (v,v)$ built-in to any Cartesian category. Thus, to copy the single parameter space $\R^{P_0}$ to all tasks, we simply apply $\delta_{P_0}$ $N$ times. We denote this $N$-fold application $\delta_{P_0}^N$. To recover the problem in \eqref{eq:mtl_gd}, all we must do is precompose our monoidal product of task specific objectives with the $N$-fold copy, yielding the morphism in $\OptEuc$

\begin{equation}\label{eq:mtl_morph}
    \begin{array}{ll}
         \Bigl( \left( \bigotimes_{i = 1}^N F^i \right) \circ \delta_{P_0}^N, O(\langle b_0, b_1 \rangle) \circ \left( 0 + \sum_{i=1}^N \Li \circ \pi^i \right) \Bigr) & = \\
         \Bigl( \left( \bigotimes_{i = 1}^N F^i \right) \circ \delta_{P_0}^N, \Ld \maps (W_0,W_1,\dots,W_N)\mapsto\sum_{i=1}^N \Li(W_0,W_i) \Bigr)
         
    \end{array}
\end{equation}

which can be pictured diagramatically as:

\[\begin{tikzcd}
	&& {\mathbb{R}^{P_0}\times \prod_{i = 1}^N \mathbb{R}^{P_i}} \\
	\\
	& {\mathbb{R}^{P_0}} & {\mathbb{R}^{P_0}\times \left( \prod_{i = 1}^N \mathbb{R}^{P_0} \times \mathbb{R}^{P_i} \right)} & {\prod_{i = 1}\mathbb{R}^{P_0} \times \mathbb{R}^{P_i}} \\
	{\mathbb{R}^{P_0}} && {\prod_{i = 1}^N\mathbb{R}^{P_0}} && {\prod_{i = 1}\mathbb{R}^{P_i}}
	\arrow["{\mathrm{id}}", from=3-2, to=4-1]
	\arrow["\Delta"', from=3-2, to=4-3]
	\arrow["{\prod_{i=1}^N \pi_0^i}", from=3-4, to=4-3]
	\arrow["{\prod_{i=1}^N \pi_1^i}"', from=3-4, to=4-5]
	\arrow["{{\hat{\pi}_1}}", from=3-3, to=3-4]
	\arrow["{{\hat{\pi}_0}}"', from=3-3, to=3-2]
	\arrow["{{b_1}}", curve={height=-12pt}, from=1-3, to=3-4]
	\arrow["{{b_0}}"', curve={height=12pt}, from=1-3, to=3-2]
	\arrow["{{\langle b_0,b_1\rangle}}"{description}, dashed, from=1-3, to=3-3]
\end{tikzcd}\]

In \eqref{eq:mtl_morph}, the computation of the pullback has effectively filtered the multiple copies of $\R^{P_0}$ and left us with the loss we want. Note that the objective function $0$ of the comultiplication morphism does not contribute to the overall objective. 

\subsection{Distributed Optimization via Functoriality of Gradient Descent}
\SetKwComment{Comment}{/* }{ */}
\begin{algorithm}\label{alg:mtl}
    \caption{Distributed MTL}
    \KwIn{Initial shared weights $W_0^0\in\R^{P_0}$, and task-specific weights $W_i^0\in\R^{P_i}$ for $i\in \{1,\dots,N\}$}
    \KwIn{A positive real learning rate $\gamma$}
    $W_i\gets W_i^0$ for $i\in \{0,\dots,N\}$\;
    \While{a stopping criterion is not reached}{
        Compute each iteration of the following for-loop in parallel\;
        \For{$i\gets 1$ \KwTo $N$}{
            $\texttt{grad\_}W_{0,i}\gets \nabla_{W_0} \mathcal{L}(W_0,W_i)$\;
            $\texttt{grad\_}W_i\gets \nabla_{W_i} \mathcal{L}(W_0,W_i)$\;
            $W_i\gets W_i - \gamma * \texttt{grad\_}W_i$\Comment*[r]{Update non-shared weights}
        }
        $\texttt{grad\_}W_0\gets \sum_{i=1}^N \texttt{grad\_}W_{0,i}$\Comment*[r]{Gradient of shared weights is sum of gradients from each task-specific learner}
        $W_0\gets W_0 - \gamma*\texttt{grad\_}W_0$\;
    }
    \Return{$W_0,W_1,\dots,W_N$}
\end{algorithm}

Having represented \eqref{eq:mtl_gd} as a morphism in $\OptEuc$, we can create a gradient descent based optimizer by applying Cor. \ref{thm:GD}. In particular we aim for a distributed gradient descent scheme. Exploiting the functoriality of $\GDEuc$, we prefer to compute the gradient for each morphism separately and then compose systems. The generalized gradient functor is illustrated in Section \ref{sec:intro}, Figure \ref{fig:mtl}.

Algorithm \ref{alg:mtl} describes the distributed scheme for multitask learning derived by applying $\GD$ to the individual subproblems and composing the resulting dynamical systems. Note that there is value in distributing computation this way as computing the gradients of weights will typically involve non-trivial backpropogation over potentially large networks. 
With this, we have successfully recovered hard parameter sharing for multitask learning and shown how to make a gradient descent optimizer that can be implemented in a distributed fashion.


\section{Discussion and Future Work}
Taking stock, it is fair to wonder what has been gained by specifying parameter sharing via decorated spans. One classic argument for the utility of applied category theory is the tight link between graphical problem depiction and its realization in a mathematical or computational data structure. Figure \ref{fig:mtl} contains two string diagrams which visually capture the structure of multitask learning objectives and optimizers respectively, enabling better communication of models between machine learning practitioners. Moreover, the relative ease with which one can describe compositional objectives supports more complex problem structures. Objectives with traditional regularization terms can be trivially described as compositional objectives, and soft parameter sharing should be fairly easily accommodated via the addition of penalization terms. More exotic hierarchical parameter sharing (see e.g. \cite{sogaard2016deep}) should be readily modeled using our approach to MTL. There is a computational benefit too; distributed optimization schemes become easier to implement with a rigorous definition of problem structure. Finally the generality of the techniques discussed allows the parameter sharing pattern of the given morphism to be applied to settings beyond smooth real-valued functions and traditional gradient descent.

As for the general framework defined in Sec. \ref{sec:GGD}, we believe there are interesting connections with previous work on CRDCs. Perhaps the most obvious direction for future work would be to understand the connections with the $\Para$ and Lens constructions of \cite{cruttwell2022categorical}. In particular, we hope that a unified treatment of both the compositional structure of models and the compositional structure of objectives can be achieved, perhaps using double categories.
Naturally another direction of future work is to find more examples of CRDCs and study not only learning problems in those settings, but other optimization problems such as those arising in operations research.

\section{Conclusion}
In summary, we have presented hypergraph categories $\OptC$ and $\DynamC$ of open objectives and dynamical systems respectively, and shown that there exists a hypergraph functor between them that maps a generalized objective to its corresponding gradient descent optimizer. Both $\OptC$ and $\DynamC$ are constructed with respect to an underlying Cartesian reverse derivative category, allowing for composition of optimization problems that are defined on a general optimization domain. Such compositionality induces a template for distributed optimization, and provides an accompanying graphical syntax that facilitates communication and sharing. Lastly, to provide evidence that existing machine learning paradigms can be described by our compositional framework, we described parameter sharing models for multitask learning using the hypergraph structure of $\OptC$, and highlighted the distributed gradient descent algorithm induced by this structure. 

\section{Acknowledgements}

The authors would like to thank Jules Hedges, Simon Willerton, Richard Samuelson, Matthew Hale, Evan Patterson, and David Spivak for helpful conversations. Tyler Hanks was supported by the National Science Foundation Graduate
Research Fellowship Program under Grant No. DGE-1842473. Any opinions, findings,
and conclusions or recommendations expressed in this material are those of the author(s)
and do not necessarily reflect the views of the NSF. James Fairbanks was supported by DARPA under award no. HR00112220038 and ONR under award no. N000142312339.

\bibliographystyle{eptcs}
\bibliography{main}

\appendix
\section{Proofs}\label{sec:app}
The following helper lemma will be crucial going forward.
\begin{lemma}\label{lem:helpers}
    Let $\phi\maps X\to Y$ and $\psi\maps X'\to Y'$ be arbitrary morphisms in $\mathrm{Lin}(\CC)$. Then the following equations hold:
    \begin{enumerate}
        \item $(\phi\times\psi)^\dag = \phi^\dag\times\psi^\dag$,
        \item $\pi_0\circ (\phi\times\psi)=\phi\circ\pi_0$,
        \item $\pi_1\circ (\phi\times\psi)=\psi\circ\pi_1$.
    \end{enumerate}
\end{lemma}
\begin{proof}
    For (1), we have
    \begin{equation}
        \begin{array}{ll@{}ll}
            (\phi\times\psi)^\dag & = \langle \phi\circ\pi_0,\psi\circ\pi_1\rangle^\dag & \\
             & = (\iota_0\circ\phi\circ\pi_0 + \iota_1\circ\psi\circ\pi_1)^\dag & \text{ Lemma 3 in \cite{cockett2019reverse}}\\
             & = (\iota_0\circ\phi\circ\pi_0)^\dag + (\iota_1\circ\psi\circ\pi_1)^\dag & \text{ Property of $\dag$-biproducts}\\
             &= \pi_0^\dag\circ\phi^\dag\circ\iota_0^\dag + \pi_1^\dag\circ\psi^\dag\circ\iota_1^\dag & \text{ Contravariant functoriality} \\
             &= \iota_0\circ\phi^\dag\circ\pi_0 + \iota_1\circ\psi^\dag\circ\pi_1 & \text{ Lemma 21 in \cite{cockett2019reverse}}\\
             &= \langle \phi^\dag\circ\pi_0,\psi^\dag\circ\pi_1\rangle = \phi^\dag\times\psi^\dag.
        \end{array}
    \end{equation}
    For (2), let $(x,x')\in X\times X'$ be generalized elements. Then
    \begin{equation}
        \pi_0\circ(\phi\times\psi)(x,x') = \phi(x) = \phi(\pi_0(x,x')).
    \end{equation}
    A symmetric argument holds for (3).
\end{proof}

\obj*
\begin{proof} 
    $O$ is plainly functorial as it is just the contravariant hom functor $\CC(-,R)$ restricted to act only on the linear maps of $\CC$. We still need to verify the symmetric lax monoidal axioms. For naturality of the product comparison, we need the following diagram to commute for all objects $X,X',Y,Y'$ and linear maps $\phi\maps X'\to X$ and $\psi\maps Y'\to Y$:
\[\begin{tikzcd}
	{O(X)\times O(Y)} && {O(X')\times O(Y')} \\
	{O(X\times Y)} && {O(X'\times Y')}
	\arrow["{\varphi_{X,Y}}"', from=1-1, to=2-1]
	\arrow["{O(\phi)\times O(\psi)}", from=1-1, to=1-3]
	\arrow["{\varphi_{X',Y'}}", from=1-3, to=2-3]
	\arrow["{O(\phi\times\psi)}"', from=2-1, to=2-3]
\end{tikzcd}\]
    Letting $f\maps X\to R$ and $g\maps Y\to R$ be arbitrary, following the top path yields the morphism 
    \begin{equation}\label{eq:top}
        f\circ\phi\circ\pi_0 + g\circ\psi\circ\pi_1
    \end{equation}
    while following the bottom path yields
    \begin{equation}\label{eq:bot}
        (f\circ\pi_0+g\circ\pi_1)\circ(\phi\times\psi)=f\circ\pi_0\circ(\phi\times\psi)+g\circ\pi_1\circ(\phi\times\psi),
    \end{equation}
    where the equality comes from $\textbf{LA.1}.$ Now let $x\in X'$ and $y\in Y'$ be generalized elements. Then applying \eqref{eq:top} to $(x,y)$ gives $f(\phi(x))+g(\psi(y))$ while applying \eqref{eq:bot} to $(x,y)$ gives
    \begin{equation}
        f(\pi_0(\phi(x),\psi(y))) + g(\pi_1(\phi(x),\psi(y))) = f(\phi(x))+g(\psi(y)),
    \end{equation}
    as desired. Commutativity of the symmetry, associativity, and unitality diagrams all follow from the fact that the hom-set $\CC(X'\times Y', R)$ is a commutative monoid.
\end{proof}

\opt*
\begin{proof}
    For functoriality of $D$, let $\phi\maps X\to Y, \psi\maps Y\to Z$, and $v\maps Z\to Z$ be arbitrary. For preservation of composition, we have
    \begin{multline}
        D(\phi\circ\psi)(v)\coloneqq (\phi\circ\psi)^\dag\circ v\circ (\phi\circ\psi)
        = (\psi^\dag\circ\phi^\dag)\circ v\circ (\phi\circ\psi)
        \\= \psi^\dag\circ (\phi^\dag\circ v\circ \phi)\circ\psi = (D(\psi)\circ D(\phi))(v),
    \end{multline}
    where the first equality comes from contravariant functoriality of $(-)^\dag$ and the second equality comes from associativity of composition. Similarly, for preservation of identities, we have
    \begin{equation}
        D(\id_Z)(v)\coloneqq \id_Z^\dag\circ v\circ \id_Z = \id_Z\circ v\circ \id_Z = v,
    \end{equation}
    where the first equality again comes from functoriality of $(-)^\dag$.

    To verify the naturality of the product comparison, we need the following diagram to commute for all $\phi\maps X'\to X$ and $\psi\maps Y'\to Y$:
\[\begin{tikzcd}
	{D(X)\times D(Y)} && {D(X')\times D(Y')} \\
	{D(X\times Y)} && {D(X'\times Y')}
	\arrow["{\varphi_{X,Y}}"', from=1-1, to=2-1]
	\arrow["{D(\phi)\times D(\psi)}", from=1-1, to=1-3]
	\arrow["{\varphi_{X',Y'}}", from=1-3, to=2-3]
	\arrow["{D(\phi\times\psi)}"', from=2-1, to=2-3]
\end{tikzcd}\]
    For this, let $(v,w)\in D(X)\times D(Y)$ be arbitrary. Then, following the top path yields 
    \begin{equation}
        \pi_0^\dag\circ\phi^\dag\circ v\circ\phi\circ\pi_0 + \pi_1^\dag\circ\psi^\dag\circ w\circ\psi\circ\pi_1,
    \end{equation}
    while following the bottom path yields
    \begin{equation}
        \begin{array}{ll@{}ll}
            (\phi\times\psi)^\dag\circ(\pi_0^\dag\circ v\circ\pi_0 + \pi_1^\dag\circ w\circ \pi_1)\circ (\phi\times\psi) &  \\
            = (\phi\times\psi)^\dag\circ(\pi_0^\dag\circ v\circ\pi_0\circ (\phi\times\psi) + \pi_1^\dag\circ w\circ\pi_1\circ(\phi\times\psi)) & \text{ Left additivity} \\
            = (\phi\times\psi)^\dag\circ\pi_0^\dag\circ v\circ \pi_0\circ(\phi\times\psi) + (\phi\times\psi)^\dag\circ\pi_1^\dag\circ w\circ\pi_1\circ(\phi\times\psi) & \text{ Additivity of linear maps}\\
            = (\phi\times\psi)^\dag\circ\pi_0^\dag\circ v\circ \phi\circ\pi_0 + (\phi\times\psi)^\dag\circ\pi_1^\dag\circ w\circ\psi\circ\pi_1 & \text{ Lemma \ref{lem:helpers}}\\
            = (\pi_0\circ \phi\times\psi)^\dag\circ v\circ \phi\circ\pi_0 + (\pi_1\circ\phi\times\psi)^\dag\circ w\circ\psi\circ\pi_1 & \text{ Contravariant functoriality}\\
            = (\phi\circ\pi_0)^\dag\circ v\circ\phi\circ\pi_0 + (\psi\circ\pi_1)^\dag\circ w\circ \psi\circ\pi_1 & \text{ Lemma \ref{lem:helpers}}\\
            = \pi_0^\dag\circ\phi^\dag\circ v\circ \phi\circ \pi_0 + \pi_1^\dag\circ\psi^\dag\circ w \circ \psi\circ\pi_1,
        \end{array}
    \end{equation}
    as desired. The symmetric monoidal coherence axioms again follow from the commutative monoid structure of hom-sets.
\end{proof}

\ggd*
\begin{proof}
    Let $\gd\coloneqq \gd_\CC, O\coloneqq O_\CC^R,$ and $D\coloneqq D_\CC$. For naturality, we need to verify that the following diagram commutes for all objects $X,Y\in \lin(\CC)$ and linear maps $\phi\maps Y\to X$:
\[\begin{tikzcd}
	{O(X)} && {O(Y)} \\
	\\
	{D(X)} && {D(Y)}
	\arrow["{O(\phi)}", from=1-1, to=1-3]
	\arrow["{\gd_X}"', from=1-1, to=3-1]
	\arrow["{\gd_Y}", from=1-3, to=3-3]
	\arrow["{D(\phi)}"', from=3-1, to=3-3]
\end{tikzcd}\]
    Let $f\maps X\to R$ be an arbitrary objective. Then, following the top path yields the optimizer
    \begin{equation}\label{eq:top1}
        -\rev[f\circ\phi]_1\coloneqq -\rev[f\circ\phi]\circ\langle \id_Y,1_{YR}\rangle.
    \end{equation}
    Likewise, following the bottom path yields
    \begin{equation}\label{eq:bot1}
        -\phi^\dag\circ\rev[f]_1\circ\phi\coloneqq -\phi^\dag\circ\rev[f]\circ\langle \id_X,1_{XR}\rangle\circ\phi.
    \end{equation}
    Now consider applying \eqref{eq:top1} to a generalized element $y\in Y$:
    \begin{equation}
        \begin{array}{ll@{}ll}
            -\rev[f\circ\phi]\circ\langle \id_Y, 1_{YR}\rangle(y) & = -\rev[f\circ\phi](y,1_R) \\
             & = -\rev[\phi](y,\rev[f](\phi(y),1_R)) & \text{ Chain rule}\textbf{ RD.5}\\
             & = -\phi^\dag(\rev[f](\phi(y),1_R)) & \text{ Linearity.}\\
        \end{array}
    \end{equation}
    Finally, applying \eqref{eq:bot1} to the same element gives
    \begin{equation}
        -\phi^\dag(\rev[f](\langle \id_X, 1_{XR}\rangle(\phi(y)))) = -\phi^\dag(\rev[f](\phi(y),1_R))
    \end{equation}
    as desired. To verify the transformation is monoidal, we must show that the following diagrams commute:
\[\begin{tikzcd}
	{O(X)\times O(Y)} && {D(X)\times D(Y)} & {*} && {O(*)} \\
	{O(X\times Y)} && {D(X\times Y)} &&& {D(*)}
	\arrow["{\gd_X\times\gd_Y}", from=1-1, to=1-3]
	\arrow["{\varphi_{X,Y}}"', from=1-1, to=2-1]
	\arrow["{\varphi_{X,Y}}", from=1-3, to=2-3]
	\arrow["{\gd_{X\times Y}}"', from=2-1, to=2-3]
	\arrow["{\varphi_0}", from=1-4, to=1-6]
	\arrow["{\varphi_0}"', from=1-4, to=2-6]
	\arrow["{\gd_*}", from=1-6, to=2-6]
\end{tikzcd}\]
    Let $g\maps Y\to R$ be another arbitrary objective. Following the top path of the product comparison diagram yields the system
    \begin{equation}
        -\pi_0^\dag\circ\rev[f]_1\circ\pi_0 -\pi_1^\dag\circ\rev[g]_1\circ\pi_1
    \end{equation}
    while following the bottom path gives
    \begin{equation}
    \begin{array}{ll@{}ll}
        -\rev[f\circ\pi_0+g\circ\pi_1]_1 &= -\rev[f\circ\pi_0+g\circ\pi_1]\circ\langle \id_{X\times Y}, 1_{X\times Y, R}\rangle \\
        &= -(\rev[f\circ\pi_0] + \rev[g\circ\pi_1])\circ\langle \id_{X\times Y}, 1_{X\times Y, R}\rangle & \textbf{ RD.1}\\
        &=-\rev[f\circ\pi_0]\circ \langle \id_{X\times Y}, 1_{X\times Y, R}\rangle -\rev[g\circ\pi_1]\circ \langle \id_{X\times Y}, 1_{X\times Y, R}\rangle & \textbf{ LA.1}\\
        
        &=-\rev[f\circ\pi_0]_1 - \rev[g\circ\pi_1]_1.
    \end{array}
    \end{equation}
    These can be shown equivalent following the same reasoning as for naturality, namely by applying the chain rule. The unit comparison diagram commutes trivially as $D(*)$ is terminal.
\end{proof}

\end{document}